\documentclass[11pt, letterpaper, draft]{article}

\usepackage{amsfonts}
\usepackage{amssymb}
\usepackage{amsmath}
\usepackage{amsthm}
\usepackage{latexsym}

\newtheorem{thm}{Theorem}[section]
\newtheorem{lemma}[thm]{Lemma}

\newtheorem*{thm*}{Theorem}

\theoremstyle{definition}
\newtheorem{remark}[thm]{Remark}

\newcommand{\lp}{\left(}
\newcommand{\rp}{\right)}
\newcommand{\texto}[1]{\quad\mbox{#1}\quad}

\newcommand{\R}{\mathbb{R}}
\newcommand{\divergence}{\text{div}}
\newcommand{\g}{g_{eq}}

\newcommand{\trace}{\text{trace}}

\newcommand{\abs}[1]{\left\lvert #1\right\rvert}

\newcommand{\be}{\begin{equation}}
\newcommand{\ee}{\end{equation}}
\newcommand{\bee}{\begin{equation*}}
\newcommand{\eee}{\end{equation*}}
\newcommand{\bea}{\begin{eqnarray}}
\newcommand{\eea}{\end{eqnarray}}
\newcommand{\bs}{\begin{split}}
\newcommand{\es}{\end{split}}

\begin{document}

\title{\bf Fractional Laplacian in Conformal Geometry}
 \author{Chang, Sun-Yung Alice\thanks{The research of the first author is partially supported by NSF through grant
DMS-0758601; the first author also gratefully acknowledges partial support
from the Minerva Research Foundation and the Charles Simonyi Endowment fund during the academic year 08-09 while visiting Institute of Advanced Study.} \\Princeton University \and Mar\'ia del Mar Gonz\'alez \thanks{M.d.M Gonz\'alez is supported by Spain Government project MTM2008-06349-C03-01 and
GenCat 2009SGR345. Also, by NSF grant DMS-0635607, while staying at the Institute for Advanced Study in 2009.}\\Univ. Polit\`ecnica de Catalunya}
\date{}
\maketitle

\abstract{In this note, we study the connection between the fractional
Laplacian operator that appeared in the recent work of Caffarelli-Silvestre and a class of conformally covariant operators in conformal geometry.}


\section{Introduction}

\setcounter{equation}{00}

In recent years, there has been independent study of fractional order operators by two different group of mathematicians. On one hand, there are extensive works that study properties of fractional Laplacian operators as non-local operators together with applications to free-boundary value problems and non-local minimal surfaces - by Caffarelli-Silvestre \cite{Caffarelli-Silvestre}, and many others (see the related articles \cite{Caffarelli-Salsa-Silvestre:regularity-estimates}, \cite{Caffarelli-Souganidis:convergence-nonlocal}, \cite{Caffarelli-Roquejoffre-Savin:minimal-surfaces}, \cite{Caffarelli-Vasseur:drift-diffusion}); on the other hand, there is the work of Graham-Zworski \cite{Graham-Zworski:scattering-matrix}, (see also \cite{Guillarmou-Qing:spectral-characterization},  \cite{Gonzalez-Mazzeo-Sire}, \cite{Qing-Raske:positive-solutions}, for instance), that study a general class of conformally covariant operators ($P_{\gamma})$, parameterized by a real number $\gamma$ and defined on the boundary of a conformally compact Einstein manifold, and which includes the fractional Laplacian operators as a special case when the boundary is the Euclidean space setting as boundary of the hyperbolic space. In this note, we will clarify the connection between the work of these two groups.\\

This paper is organized as follows: in section 2 of the paper we will briefly describe the the work of Graham-Zworski \cite{Graham-Zworski:scattering-matrix} and the notion of the fractional Paneitz operator $P_{\gamma}$.

In section 3, we will illustrate in theorem \ref{thm-zero} that in the case when the fractional Laplacian operator ${(-\Delta)}^{\gamma}$ is defined on the Euclidean space ${\mathbb R}^n $ and $\gamma\in(0,1)$, the operator agrees with $P_{\gamma}$ on the hyperbolic space. This is done by applying the extension theorem of Caffarelli-Silvestre \cite{Caffarelli-Silvestre} on a characterization of the fractional Laplacian operator ${(-\Delta)}^{\gamma}$.
We then extend the work of Caffarelli-Silvestre \cite{Caffarelli-Silvestre} and the identification to $P_{\gamma}$ as ${(-\Delta)}^{\gamma}$ on  $\mathbb R^n$ for more general exponents $\gamma \in \lp 0, \frac n 2\rp$.
To achieve this, we first show in theorem \ref{thm-extension1} that there are two different ways to define $P_{\gamma}$ operator when $\gamma>1$, one can define it through the scattering matrix as in the original work of Graham-Zworski, or one can define it through an iterated process based on the work of Caffarelli-Silvestre; and two definitions agree (the relationship is illustrated in the claim \eqref{equation*} in the proof of theorem \ref{thm-extension1} and equations
\eqref{formula11} and \eqref{formula15}). Finally, we apply theorem
\ref{thm-extension1} to generalize the extension theorem of Caffarelli-Silvestre from $\gamma\in(0,1)$ to $\gamma\in\lp 0,\frac n 2\rp $ (this is theorem \ref{thm-extension2}).

In section 4 of the paper, we will discuss the extension theorem in the general setting of conformally compact Einstein manifolds. One point we make is that by choosing a suitable defining function (in lemma \ref{lemma-special-defining}) which is related to the eigenfunctions of Laplacian of the Einstein metric, equation (\ref{equation-extension}) of the extension theorem on a general conformally compact Einstein manifold (theorem \ref{thm-new-extension}) is the same as the extension theorem on hyperbolic space studied in section 3; it is a pure divergence equation, degenerate elliptic, with a weight in the Muckenhoupt class $A_2$.

Finally, in section 5, we will discuss the extension theorem on a general asymptotically hyperbolic manifold; where the boundary manifold may not be the totally geodesic boundary of an asymptotically hyperbolic space. In particular, it provides a natural way to define the conformal fractional Laplacian on the boundary of any compact manifold, and a related fractional order curvature $Q_\gamma$. In this case, it is interesting to see that the statement of extension theorem (theorem 5.1)  breaks down at $\gamma = \frac 1 2$ and at $\gamma > \frac 1 2$. In particular, there is an extra term in the the expression of $P_{\frac 1 2}$ which is the mean curvature of the
boundary manifold. This dichotomy has already appeared in other settings (Caffarelli-Souganidis \cite{Caffarelli-Souganidis:convergence-nonlocal}, for instance), and illustrates the fact that when $\gamma\in\lp 0,\frac 1 2\rp$, the operator presents a stronger non-local behavior than when $\gamma\in\lp\frac 1 2,1\rp$.\\

The authors wish that the results in this paper will lead to extension of the works of Caffarelli and others on fractional free boundary value problems and fractional mean curvature surfaces to general manifold settings. In particular, we pose the question of finding the relationship between the fractional order mean curvature very recently defined (see Caffarelli-Roquejoffre-Savin \cite{Caffarelli-Roquejoffre-Savin:minimal-surfaces}, Caffarelli-Valdinoci \cite{Caffarelli-Valdinoci:non-local}), and our notion of $Q_\gamma$ curvature coming from the fractional Paneitz operator. This seems to be a very interesting open question.


\section{Background}

\setcounter{equation}{00}

First  we review the connection between scattering theory on conformally compact Einstein manifolds and conformally invariant objects
on their boundaries at infinity. The main references are Graham-Zworski \cite{Graham-Zworski:scattering-matrix} and Fefferman-Graham \cite{Fefferman-Graham:largo}.

Let $M$ be a compact manifold of dimension $n$ with a metric $\hat g$. Let $\bar X^{n+1}$ be a compact manifold of dimension $n+1$ with boundary $M$, and denote by $X$ the interior of $\bar X$. A function $\rho$ is a defining function of $\partial X$ in $X$ if
$$\rho>0 \mbox{ in } X, \quad \rho=0 \mbox{ on }\partial X,\quad d\rho\neq 0 \mbox{ on } \partial X. $$
We say that $g^+$ is a conformally compact (c.c.) metric on $X$ with conformal infinity $(M,[\hat g])$ if there exists a defining function $\rho$ such that the manifold $(\bar X,\bar g)$ is compact for $\bar g=\rho^2 g^+$, and $\bar g|_M\in [\hat g]$. If, in addition $(X^{n+1}, g^+)$ is a conformally compact manifold and $Ric[g^+] = -ng^+$, then we call $(X^{n+1}, g^+)$ a conformally compact Einstein manifold.

Given a conformally compact, asymptotically hyperbolic manifold $(X^{n+1}, g^+)$ and
a representative $\hat g$ in $[\hat g]$ on the conformal infinity $M$, there is a uniquely defining
function $\rho$ such that, on $M \times (0,\delta)$ in $X$, $g^+$ has the normal form
$g^+ = \rho^{-2}(d\rho^2 + g_\rho)$ where $g_\rho$ is a one parameter family of metrics on $M$ satisfying $g_\rho|_{M}=\hat g$. Moreover, $g_\rho$ has an asymptotic expansion which contains only even powers of $\rho$, at least up to degree $n$.\\

It is well known  (c.f. Mazzeo-Melrose \cite{Mazzeo-Melrose:meromorphic-extension}, Graham-Zworski \cite{Graham-Zworski:scattering-matrix}) that, given $f\in\mathbb C^\infty(M)$ and  $s\in\mathbb C$, the eigenvalue problem
\be\label{equation-GZ}
-\Delta_{g^+} u-s(n-s)u=0,\quad\mbox{in } X
\ee
has a solution of the form
\be\label{general-solution}u = F \rho ^{n-s} + H\rho^s,\quad F,H\in\mathcal C^\infty(X),\quad F|_{\rho=0}=f,\ee
for all $s\in\mathbb C$ unless $s(n-s)$ belongs to the pure point spectrum of $-\Delta_{g^+}$.
Now, the scattering operator on $M$ is defined as $S(s)f = H|_M$, it is a meromorphic family of pseudo-differential operators in $Re(s)>n/2$. The values $s = n/2, n/2 +1 , n/2 + 2, \ldots$ are simple poles of finite rank, these are known as the trivial poles; $S(s)$ may have other poles, however, for the rest of the paper we assume that we are not in those exceptional cases.\\

We define the conformally covariant fractional powers of the Laplacian as follows: for $s=\frac{n}{2}+\gamma$, $\gamma\in \lp 0,\frac{n}{2}\rp$, $\gamma\not\in \mathbb N$, we set
\be\label{P-operator} P_\gamma[g^+,\hat g] := d_\gamma S\lp\frac{n}{2}+\gamma\rp,\quad d_\gamma=2^{2\gamma}\frac{\Gamma(\gamma)}{\Gamma(-\gamma)}.\ee
With this choice of multiplicative factor, the principal symbol of $P_\gamma$ is exactly the principal symbol of the fractional Laplacian $(-\Delta_{\hat g})^{\gamma}$, precisely, $\abs{\xi}^{2\gamma}$.
We thus have that $P_\gamma\in  (-\Delta_{\hat g})^\gamma+\Psi_{\gamma-1}$, where we denote by $\Psi_m$ to be the set of pseudo-differential operators on $M$ of order $m$.

The operators $P_\gamma[g^+,\hat g]$ satisfy an important conformal covariance property (see \cite{Graham-Zworski:scattering-matrix}). Indeed, for a conformal change of metric
\be\label{change-metric}\hat g_{v}=v^{\frac{4}{n-2\gamma}}\hat g,\quad v>0,\ee
we have that
\be\label{confomral-invariance}P_\gamma[g^+,\hat g_v]\phi = v^{-\frac{n+2\gamma}{n-2\gamma}} P_\gamma[g^+,\hat g] \lp v\phi\rp,\ee
for all smooth functions $\phi$.

We define the $Q_\gamma$ curvature of the metric associated to the functional $P_\gamma$, to be
\be\label{Q-curvature}Q_\gamma[g^+,\hat g]:= P_\gamma[g^+,\hat g](1).\ee
In particular, for a change of metric as \eqref{change-metric}, we obtain the equation for the $Q_\gamma$ curvature:
$$P_{\gamma}[g^+,\hat g](v)=v^{\frac{n+2\gamma}{n-2\gamma}} Q_\gamma[g^+,\hat g_v].$$

When $\gamma$ is an integer, say $\gamma=k$, $k\in\mathbb N$, a careful study of the poles of $S(s)$ allows to define $P_{k}$. Indeed,
$$Res_{s=n/2 +k}S(s) = c_k P_{k},\quad  c_k = (-1)^k[2^{2k}k!(k - 1)!]^{-1}.$$
These are the conformally invariant powers of the Laplacian constructed by Graham-Jenne-Mason-Sparling \cite{Graham-Jenne-Mason-Sparling}, Fefferman-Graham \cite{Fefferman-Graham:largo}, that are local operators, and satisfy
$$P_k=(-\Delta)^k + \mbox{ lower order terms}.$$
In particular, when $k=1$ we have the conformal Laplacian,
$$P_1=-\Delta +\frac{n-2}{4(n-1)} R,$$
and when $k=2$, the Paneitz operator (c.f. \cite{Paneitz:published})
$$P_2=(-\Delta)^2 +\delta\lp a_n Rg+b_n Ric\rp d+\tfrac{n-4}{2}Q_2.$$
Finally, we note that another realization of the conformal fractional powers of the Laplacian was given by Peterson in \cite{Peterson:conformally-covariant}, through an analytic continuation argument from the differential operators $P_k$.\\


\section{The extension problem on hyperbolic space}\label{extension}

\setcounter{equation}{00}

Given $\gamma\in\mathbb R$, the fractional Laplacian on $\mathbb R^n$, denoted as $(-\Delta_x)^\gamma$, for a function $f:\R^n\to \R$ is defined as a pseudo-differential operator by
$$\widehat{(-\Delta_x)^\gamma} f(\xi)=\abs{\xi}^{2\gamma}\hat f(\xi),$$
i.e, its principal symbol is $\abs{\xi}^{2\gamma}$. It can also be written as the singular integral (suitably regularized)
$$(-\Delta_x)^\gamma f(x)=C_{n,\gamma}\int_{\R^n}\frac{f(x)-f(\xi)}{\abs{x-\xi}^{n+2\gamma}}\;d\xi.$$

\medskip

Caffarelli-Silvestre have developed in \cite{Caffarelli-Silvestre} an equivalent definition, in the case $\gamma\in(0,1)$, using an extension problem to the upper half-space $\R^{n+1}_+$. For a function $f:\R^n\to \R$, we construct the extension $U:\R^n\times [0,+\infty)\to \R$, $U=U(x,y)$, as the solution of the equation
\be\label{equation-CS}
\left\{\begin{split}
\Delta_x U+\frac{a}{y}\partial_y U+\partial_{yy} U=& \;0, \quad \quad x\in\R^n,\; y\in[0,+\infty),\\
U(x,0)=& f(x),\quad x\in\mathbb R^n,
\end{split}\right.
\ee
where
$$\gamma=\frac{1-a}{2}.$$
Note that equation \eqref{equation-CS} can be written as a divergence equation
$$\divergence(y^a\nabla U)=0\quad\mbox{in }\mathbb R^{n+1}_+,$$
which is degenerate elliptic.
Then the fractional Laplacian of $f$ can be computed as
\be\label{Neumann-CS}(-\Delta_x)^\gamma f=\frac{d_\gamma}{2\gamma}\lim_{y\to 0} y^a \partial_y U,\ee
where the constant $d_\gamma$ is defined in \eqref{P-operator}. We are indeed looking at a non-local Dirichlet to Neumann operator. To finish, just note that the Poisson kernel for the fractional Laplacian $(-\Delta)^\gamma$ in $\mathbb R^n$ is
$$K_\gamma(x,y)=c_{n,\gamma}\frac{y^{1-a}}{\lp\abs{x}^2+\abs{y}^2\rp^{\frac{n+1-a}{2}}},$$
and thus $U=K_\gamma *_x f$.\\

The main observation we make in this section is that, in the case $M=\R^n$ and $X=\R^{n+1}_+$ with coordinates $x\in\R^n$, $y>0$, endowed the hyperbolic metric $g_{\mathbb H}=\frac{dy^2+\abs{dx}^2}{y^2}$, the scattering operator is nothing but the Caffarelli-Silvestre extension problem for the fractional Laplacian when $\gamma\in(0,1)$. After that, we give the generalization of the result by Caffarelli-Silvestre for other exponents $\gamma\in \lp 0,\frac{n}{2}\rp\backslash\mathbb N$. For simplicity, we write $P_\gamma:=P_\gamma[g_{\mathbb H},\abs{dx}^2]$, where $\abs{dx}^2$ is the Euclidean metric on $\mathbb R^n$.

\begin{thm} \label{thm-zero}
Fix $\gamma\in(0,1)$ and $f$ a smooth function defined on $\R^n$. If $U$ is a solution of the extension problem \eqref{equation-CS}, then $u=y^{n-s}U$ is a solution of the eigenvalue problem \eqref{equation-GZ} for $s=n/2+\gamma$,  and moreover,
\be\label{first-equality}P_\gamma f=\frac{d_\gamma}{2\gamma}\,\lim_{y\to 0} \lp y^a\partial_y U\rp=(-\Delta_x)^\gamma,\ee
where $a=1-2\gamma$, $P_\gamma:=P_\gamma[g_{\mathbb H},\abs{dx}^2]$, and the constant $d_\gamma$ is defined in \eqref{P-operator}.
\end{thm}

\begin{proof}
Fix $f$ on $\mathbb R^n$ and let $u$ be a solution of the scattering problem
\be\label{formula1}-\Delta_{\mathbb H}u-s(n-s)u=0\mbox{ in }X.\ee
We know that $u$ can be written as
\be\label{formula4}u=y^{n-s}F+y^s H,\ee
where $F|_{y=0}=f$ and $S(s)f
=h$ for $h=H|_{y=0}$.
Moreover,
\be\label{formula5}
F(x,y)=f(x)+f_2(x)y^2+o(y^2) \texto{and} H(x,y)=h(x)+h_2(x)y^2+o(y^2).
\ee

On the other hand, the conformal Laplacian operator for a Riemannian metric $g$ in a manifold $X$ of dimension $N=n+1$ is defined as
$$L_g=-\Delta_g+\frac{N-2}{4(N-1)}R_g.$$
For the hyperbolic metric, $R_{g_{\mathbb H}}=-n(n+1)$, so that
\be\label{conformal-laplacian2}
L_{g_{\mathbb H}}=-\Delta_{g_{\mathbb H}}-\tfrac{n^2-1}{4}\ee
Then, from \eqref{formula1} we can compute
\be\label{formula40}
0=-\Delta_{g_{\mathbb H}} u-s(n-s)u=L_{g_{\mathbb H}} u+\lp\gamma^2-\tfrac{1}{4}\rp u=y^{\frac{n+3}{2}} L_{g_{eq}} \lp y^{-\frac{n+1}{2}}u\rp+\lp\gamma^2-\tfrac{1}{4}\rp u.\ee
where in the last equality we have used the conformal covariant property of the conformal Laplacian for the change of metric $g_{eq}=y^2 g_{\mathbb H}$:
\be\label{conformal-laplacian1}
L_{g_{\mathbb H}}(\psi)=y^{\frac{n+3}{2}}L_{\g}\lp y^{-\frac{n-1}{2}}\psi\rp.\ee
Next, we change $u=y^{n-s}U$, and note that
\be\label{conformal-laplacian3}L_{g_{eq}}=-\Delta=-\Delta_{x}-\partial_{yy},\ee
so it follows that
\be\label{formula3}\begin{split}L_{g_{eq}}&\lp y^{\frac{n+1}{2}-s}U\rp\\
&=-y^{\frac{n+1}{2}-s}\left[\Delta_x U +\partial_{yy}U+\frac{a}{y}\partial_y U+\lp\tfrac{n+1}{2}-s\rp\lp\tfrac{n+1}{2}-s-1\rp\frac{U}{y^2} \right].\end{split}\ee
Substituting \eqref{formula3} into \eqref{formula40}, we observe that with the choice of $s=\frac{n}{2}+\gamma$ and $a=1-2\gamma$ we arrive at
$$\Delta_x U +\partial_{yy}U+\frac{a}{y}\partial_y U=0,$$
as we wished.

For the second part of the lemma, note that
\be\label{eq4}P_\gamma f=d_\gamma S\lp\tfrac{n}{2}+\gamma\rp=d_\gamma\;h,\ee
where $h$ is given in \eqref{formula5}. On the other hand, we also have that
$$U=y^{s-n}u=F+y^{2s-n}H,$$
and thus, looking at the orders of $y$ in \eqref{formula5}, we can conclude that the limit
\be\label{eq3}\lim_{y\to 0} y^{a}\partial_y U\ee
exists and equals $h$ times the constant $2\gamma$.
The lemma is proven by comparing \eqref{eq3} together with \eqref{eq4}, with the Caffarelli-Silvestre construction for the fractional Laplacian as given in \eqref{Neumann-CS}.
\end{proof}

\bigskip

The next step is to generalize theorem \ref{thm-zero} to other non-integer exponents $\gamma$. To do this, we will first establish in theorem 3.2 below that there are two ways to define the operator $P_{\gamma}$ when $\gamma >1$, and the two definitions agree. We will then apply theorem 3.1 generalize  theorem \ref{thm-zero} to $\gamma >1$.

The first way to define $P_{\gamma}$ is the original definition of  Graham Zworski \cite{Graham-Zworski:scattering-matrix}, that is to define it using scattering matrix; thus when $\gamma >1$, $P_{\gamma} f$ agrees with a higher order term in the power series expansion of the solution of a second order equation (e.g. equation \eqref{formula15});
 the other is to define it by iterating the work of Caffarelli-Silvestre resulting in a PDE of order higher than
2, one can represent $P_{\gamma} f_0 $ as a lower order term in the power series expansion of the solution of this PDE (e.g. equation \eqref{formula11}); we will show that the two definitions agree (see  claim \eqref{equation*}).\\

First, fixed $\gamma\in\lp 0,\frac{n}{2}\rp\backslash \mathbb N$, if $\gamma=m+\gamma_0$, for $m=[\gamma]\in\mathbb N$, $\gamma_0\in(0,1)$, we can define the fractional Laplacian on $\mathbb R^n$ inductively as
\be\label{definition-Laplacian}(-\Delta_x)^\gamma= (-\Delta_x)^{\gamma_0}\circ(-\Delta_x)^m.\ee
We have:

\begin{thm} \label{thm-extension1}
For any $\gamma\in\lp 0,\frac{n}{2}\rp\backslash\mathbb N$, we have that
$$P_\gamma[g_{\mathbb H},\abs{dx}^2]=(-\Delta_x)^\gamma,$$
where the fractional conformal Laplacian $P_\gamma$ on $\mathbb R^n$ is defined as in \eqref{P-operator}.
\end{thm}

\begin{proof} The Proof is by induction on $m$.
We set up the following notation:
\bee\begin{split}
& \gamma_j=\gamma-(m-j) \texto{for} j=0,\ldots,m. \quad \mbox{Note that } j<\gamma_j<j+1, \quad \gamma_m=\gamma,\\
& s_j=\frac{n}{2}+\gamma_j,\\
& a_j=1-2\gamma_j.
\end{split}\eee
For each $j=0,\ldots,m$, we set $f_j:=(-\Delta_{x})^{m-j} f$.
Then the  eigenvalue problem
$$-\Delta_{g_\mathbb H} u_{s_j}-{s_j}(n-s_j)u_j=0,\quad \mbox{in } \mathbb H^{n+1},$$
has  a unique solution $u_j=F_j  y^{n-s_j}+H_j y^{s_j}$ satisfying $F_j|_{y=0}=f_j$.
Set $h_j=H_j|_{y=0}$. Then the scattering operator is simply
$$P_{\gamma_j} f_j=d_{\gamma_j}S(s_j)f_j=h_j.$$
We set $U_j:=y^{n-s_j} u_j$. The same computation as in the proof of theorem \ref{thm-zero} gives that $U_j$ is a solution of
\be\label{eq200}\Delta U_j+\frac{a_j}{y}\partial_y U_j=0.\ee

On the other hand, since $\gamma_0\in(0,1)$, theorem \ref{thm-zero} applied to $U_0$ implies that
$$h_0=\frac{1}{2\gamma_0}\lim_{y\to 0} \lp y^{a_0} \partial_y U_0\rp,$$
and this limit is well defined.

We already know from theorem \ref{thm-zero} that $P_{\gamma_0}=(-\Delta_x)^{\gamma_0}$. On the other hand, by construction $f_0=(-\Delta_x)^m f$. Now we claim that
\be\label{equation*}P_{\gamma_0}f_0=P_\gamma f,\ee
or equivalently, because of the definition of $P_\gamma$ as in \eqref{P-operator},
\be\label{conclusion-1}d_{\gamma_0} h_0=d_\gamma h_m.\ee
The proof of the theorem will be completed if we show claim \eqref{equation*}.\\

We claim first that
\be\label{formula20}\Delta U_m=-\frac{a_m}{1+a_{m}}U_{m-1}\quad \mbox{in }\mathbb R^{n+1}_+.\ee
Indeed, as mentioned above in \eqref{eq200}, $U_m$ is a solution of
\be\label{formula12}\Delta U_m +\frac{a_m}{y} \partial_y U_m=0.\ee
Differentiating above expression with respect to $y$, we obtain
\be\label{formula13}\Delta \lp \partial_y U_m \rp -\frac{a_m}{y^2}\partial_y U_m+\frac{a_m}{y} \partial_{yy} U_m=0.\ee
And taking the Laplacian of \eqref{formula12},
\be\label{formula14}\Delta^2 U_m +a_m\lp\frac{1}{y}\Delta_x(\partial_y U_m)+\frac{1}{y^3}\partial_y U_m-\frac{2}{y^2}\partial_{yy}U_m+\frac{1}{y}\partial_{yyy}U_m\rp=0. \ee
Substitute \eqref{formula13} into \eqref{formula14} and take into account that $a_m=a_{m-1}-2$. This proves that
\be\label{formula11}\Delta^2 U_m+\frac{a_{m-1}}{y}\partial_y(\Delta U_m)=0.\ee
Next, note from \eqref{eq200} that the function $U_{m-1}$ solves the equation
\be\label{formula15}\Delta U_{m-1}+\frac{a_{m-1}}{y}\partial_y U_{m-1}=0.\ee
Thus we see from \eqref{formula11} and \eqref{formula15} that $\Delta U_m$ and $U_{m-1}$ satisfy the same equation in $\mathbb R^{n+1}_+$. Let us compare now the boundary values. First of all, by hypothesis,
\be\label{boundary1}U_{m-1}|_{y=0}=-\Delta_x f.\ee
On the other hand, we claim that
\be\label{boundary2}\Delta U_{m}|_{y=0}=\frac{a_m}{1+a_m}\Delta_x f.\ee
The proof of this fact is a simple computation: we know that $U_{m}=F_m+y^{2s_m-n}H_{m}$ where $F_m=f+b_m y^2+O(y^3)$ and $H_m=h_m+O(y^2)$. Then
we can compute
\be\label{formula18}\Delta U_m|_{y=0}=\Delta_x f+2b_m\ee
 and $\frac{1}{y}\partial_y U_m|_{y=0}=2b_m$. From \eqref{formula12} we must have
$$b_m=\frac{\Delta_x f}{2(-1-a_m)}.$$
Equation \eqref{formula18} implies then
$$\Delta U_m|_{y=0}=\frac{a_m}{1+a_m}\Delta_x f.$$
We have shown from \eqref{boundary1} and \eqref{boundary2} that $\Delta U_m$ and $U_{m-1}$ have the same boundary values (modulo a multiplicative constant). Since, as we have mentioned, they satisfy the same second order elliptic equation, they must coincide in the whole $\mathbb R^{n+1}_+$. This is,
\be\label{formula20}\Delta U_m=-\frac{a_m}{1+a_{m}}U_{m-1},\ee
as claimed in \eqref{formula20}.\\

Now we are ready to complete the proof of the theorem. From \eqref{formula20} and \eqref{formula12} we conclude that
\be\label{formula21}U_{m-1}=\frac{1+a_m}{y}\lp\partial_y U_m\rp.\ee
We look at the asymptotic expansions for $U_m$ and $U_{m-1}$:
$$U_m=F_m+y^{2s_m-n}H_m, \quad F_m=f_m +O(y^2),\quad H_m=h_m+O(y^2),$$
so that
$$\frac{1}{y}\partial_y U_{m}=\frac{1}{y}\partial_y F_{m} +\lp 2s_m-n\rp y^{2s_{m}-n-2} H_{m}+y^{2s_m-n}\left[\frac{1}{y}\partial_y H_m\right],$$
while
$$U_{m-1}=F_{m-1}+y^{2s_{m-1}-n} H_{m-1},\quad F_{m-1}=f_{m-1}+O(y^2),\quad H_{m-1}=h_{m-1} +O(y^2).$$
Since we have the relation \eqref{formula21}, comparing the coefficients of the term $y^{2s_m-n-2}=y^{2s_{m-1}-n}$ we obtain that
$$(2s_m-n)(1+a_m)h_m=h_{m-1},$$
which is
\be\label{inductive-step}(2\gamma_m) h_m=h_{m-1}.\ee
We set
$$A_m=2^{m}\gamma_m\ldots\gamma_1,\,\mbox{ if }\, m>0,\quad A_0=1,$$
and
\be\label{constant-c}c_m=\prod_{j=1}^m (a_j+1) \,\mbox{ if }\, m>0,\quad c_0=1,\ee
which will be necessary later.
Applying \eqref{inductive-step} inductively we arrive at
\be\label{relation-1}A_m c_m h_m=h_0.\ee
But because
$$c_m A_m=\frac{d_{\gamma}}{d_{\gamma_0}},$$
then we have actually shown \eqref{conclusion-1} and the theorem is proved.
\end{proof}

Moreover, we have the following characterization for $P_\gamma:=P_\gamma[g_{\mathbb H},\abs{dx}^2]$ as a Dirichlet-to-Neumann operator, thus generalizing the result of Caffarelli-Silvestre for exponents $\gamma\in\lp 0,\frac{n}{2}\rp\backslash \mathbb N$ not necessarily less than one:

\begin{thm} \label{thm-extension2}
Fix $\gamma\in\lp 0,\frac{n}{2}\rp\backslash\mathbb N$, $a=1-2\gamma$.
Let $f$ a smooth function defined on $\R^n$, and let $U=U(x,y)$ be the solution of the boundary value problem
\bee
\left\{
\begin{split} \Delta_x U +\frac{a}{y}\partial_y U+\partial_{yy} U&=0 &\texto{in}\R^{n+1}_+,\\
U(x,0)&=f(x) &\texto{for all}x\in\R^n.
\end{split}
\right.
\eee
Then function $u:=y^{n-s}U$ is the solution of the Poisson equation on the hyperbolic space $\mathbb H^{n+1}$
\bee\left\{    \begin{split}
-\Delta_{g_\mathbb H} u-s(n-s)u&=0 \quad\mbox{in}\quad \mathbb H^{n+1},\\
u&=F y^{n-s}+Hy^s,\\
F(x,0)&=f(x).\end{split}\right.\eee
And, the following limit exists and we have the equality
\be \label{Neumann}P_\gamma f=\frac{d_\gamma}{2\gamma_0} A_m^{-1} \lim_{y\to 0}y^{a_0}\partial_y\left[y^{-1}\partial_y \lp y^{-1}\partial_y\lp\ldots y^{-1}\partial_y U\rp\rp \right],\ee
where we are taking $m+1$ derivatives in the above expression, and the constant is given by
\be\label{A_m}A_m=2^{m}(\gamma-1)\ldots(\gamma-m+1).\ee
We are using the notation $m=[\gamma]\in\mathbb N$, $\gamma_0=\gamma-m$, $a_0=1-2\gamma_0$, and the constant $d_\gamma$ as defined in \eqref{P-operator}.
\end{thm}

\begin{proof}
We keep the same notation as in the previous theorem. In this construction, $U$ is precisely $U_m$.

We would like to show first first that
\be\label{equation-11}h_0=\frac{c_{m}}{2\gamma_0}\lim_{y\to 0}y^{a_0}\overbrace{\partial_y \left[ y^{-1}\partial_y(y^{-1}\ldots (y^{-1}\partial_y U_{m}))\right]}^{m+1 \mbox{ derivatives}},\ee
where the constant  $c_m$ is defined in \eqref{constant-c}, and $P_{\gamma_0}f_0=h_0$. The proof goes by induction on $m$. The case $m=0$ is precisely the conclusion of theorem \ref{thm-zero}.
Assume that it is true for $m-1$, i.e.,
\be\label{equation-10}h_0=\frac{c_{m-1}}{2\gamma_0}\lim_{y\to 0}y^{a_0}\overbrace{\partial_y \left[ y^{-1}\partial_y(y^{-1}\ldots (y^{-1}\partial_y U_{m-1}))\right]}^{m \mbox{ derivatives}}.\ee
Now substitute \eqref{formula21} in \eqref{equation-10} and use that $c_{m}=c_{m-1}(1+a_m)$. We immediately obtain \eqref{equation-11}.

Next, we recall relation \eqref{relation-1} between $h_m$ and $h_0$. The definition of the operator $P_\gamma f= d_\gamma h_m$ and \eqref{equation-11} complete the proof of the theorem.
\end{proof}


\section{The extension problem on c.c. Einstein manifolds}\label{section-general-manifolds}

\setcounter{equation}{00}

We fix $\gamma\in\lp 0,\frac{n}{2}\rp\backslash\mathbb N$.  In this section we discuss the generalization of theorems \ref{thm-zero}, \ref{thm-extension1} and \ref{thm-extension2} from hyperbolic space to any conformally compact Einstein manifold $(X,g^+)$.

First we write the extension problem analogous to \eqref{equation-CS} in a conformally compact Einstein metric. The resulting problem \eqref{div} is still of divergence-type, degenerate elliptic, and with a weight in the Muckenhoupt class $A_2$ (c.f. \cite{Muckenhoupt:Hardy}), but some lower order terms appear - they depend on the underlying geometry. This is the content of lemma \ref{lemma-any-extension}.

Next,  we point out in lemma \ref{lemma-special-defining} that, by choosing a suitable defining function $\rho^*$, which is related to the eigenfunction of the Laplacian of the Einstein metric, the equation \eqref{equation-extension} of the extension theorem on general conformal compact Einstein manifolds is the same as the extension theorem on hyperbolic space studied in section 3, of pure divergence form.

Finally, in theorems  \ref{thm-new-extension} and \ref{thm-new-extension2}, we show how the extension problem with respect to this new defining function allows to compute the fractional Paneitz operator $P_\gamma$.

\begin{lemma} \label{lemma-any-extension}
Let $(X,g^+)$ be any conformally compact Einstein manifold with boundary $M$. For any defining function $\rho$ of $M$ in $X$, not necessarily geodesic, the equation
\be\label{GZ}-\Delta_{g^+} u-s(n-s)u=0 \quad \mbox{in } (X,g^+),\ee
is equivalent to
\be\label{div}-\divergence \lp \rho^a \nabla U\rp + E(\rho) U=0\quad \mbox{in } (X,\bar g),\ee
where
$$\bar g={\rho^2} g^+,\quad U= \rho^{s-n} u$$
and the derivatives in \eqref{div} are taken with respect to the metric $\bar g$. The lower order term is given by
\be\label{Error}
E(\rho)=-\Delta_{\bar g}\lp \rho^{\frac{a}{2}} \rp \rho^{\frac{a}{2}} +\lp\gamma^2-\tfrac{1}{4}\rp \rho^{-2+a}+\tfrac{n-1}{4n} R_{\bar g} \rho^a,\ee
or writing everything back in the metric $g^+$,
\be\label{E1}E(\rho)=-\Delta_{g^+} (\rho^{\frac{n-1+a}{2}}) \rho^{\frac{-n-3+a}{2}}-\lp \tfrac{n^2}{4}-\gamma^2\rp \rho^{-2+a}.\ee
Here we denote $s=\frac{n}{2}+\gamma$, $a=1-2\gamma$.
\end{lemma}

\begin{remark}
For the model case  $X=\mathbb R^{n+1}_+$, $M=\mathbb R^n$, $g^+=\frac{dy^2+\abs{dx}^2}{y^2}$, with the defining function $y>0$, $\bar g=dy^2+\abs{dx}^2$, it automatically follows from \eqref{Error} that
$$E(y)\equiv 0.$$
\end{remark}

\begin{proof}
Note that \eqref{GZ} is equivalent to
\be\label{GZ1} L_{g^+} u+ \lp \gamma^2-\tfrac{1}{4}\rp u=0,\ee
using the fact that for an Einstein metric $g^+$,
$$L_{g^+}=-\Delta_{g^+}-\frac{n^2-1}{4}.$$
On the other hand, the invariance of the conformal Laplacian reads:
$$L_{g^+}(\phi)=\rho^{\frac{n+3}{2}} L_{\bar g}\lp \rho^{-\frac{n-1}{2}}\phi\rp$$
for the change of metric $\bar g={\rho^2} g^+$. Thus, writing everything in terms of $\bar g$ and the new function $U=\rho^{s-n} u$, then \eqref{GZ1} is just
\be\label{GZ2} L_{\bar g} \lp \rho^{\frac{a}{2}} U\rp+\lp \gamma^2-\tfrac{1}{4}\rp \rho^{-2+\frac{a}{2}} U=0.\ee
Next, it is a straightforward computation to check that:
$$\rho^{\frac{a}{2}}\Delta_{\bar g} \lp \rho^{\frac{a}{2}} U\rp=\divergence_{\bar g} \lp \rho^a \nabla_{\bar g} U\rp+\Delta_{\bar g}\lp \rho^{\frac{a}{2}}\rp \rho^{\frac{a}{2}}U.$$
Substituting the above expression into \eqref{GZ2} gives the desired result \eqref{div}.\\
\end{proof}

Let $(X,g^+)$ be a conformally compact Einstein metric with boundary $(M,[\hat g])$. Then, lemma 2.1 in \cite{Graham:volume-area} gives that, fixed a metric $\hat g$ on the boundary $M$, there exists a unique defining function $y$ in $X$ such that in the neighborhood $M\times(0,\delta)$,  the metric splits as
\be\label{splitting} g^+=\frac{dy^2+g_y}{y^2},\ee
where $g_y$ is a one-parameter family of metrics on $M$ with $g_y|_{y=0}=\hat g$.
Moreover,
\be\label{taylor}
g_y=g^{(0)}+\frac{g^{(2)}}{2}y^2+\ldots\ee
only contains even terms up to order $n$. We write
\be\label{expansion1} g^{(0)}:=\hat g,\quad g^{(1)}:=\partial_y g_y|_{y=0}=0,\quad g^{(2)}:=\partial_{yy} g_y|_{y=0}.\ee
Set $\bar g=y^2 g^+$ and
\be\label{J}\Psi:=\partial_y \log \det(g_y)=\sum_{i,j}g_y^{ij}\partial_y (g_y)_{ij}.\ee
Note that
$$\Psi_0:=\frac{1}{2n} \Psi|_{y=0}=\frac{1}{2n}trace_{\hat g}( g^{(1)})$$ is the mean curvature of $M$ as hypersurface of $(X,\bar g)$, which is zero by \eqref{expansion1}.\\

Lemma \ref{lemma-any-extension} is true for any defining function independent of the behavior near the boundary. However, in the specific case that we have the splitting \eqref{splitting}, then we can have a more explicit expression for the lower order terms $E(\rho)$:

\begin{thm} If the defining function (denoted by $y$ in the following) is chosen such that metric splits as \eqref{splitting} in a neighborhood $M\times(0,\delta)$, then
$$E(y)=\tfrac{-n+1+a}{4} \Psi y^{a-1}=\tfrac{n-1-a}{4n} R_{\bar g} y^a \quad\mbox{in } M\times(0,\delta).$$
And, in particular,
$$\lim_{y\to 0} \frac{E(y)}{y^a}=\frac{-n+1+a}{4} \,\trace_{g^{(0)}} g^{(2)}.$$
Moreover, formula \eqref{Neumann} for the calculation of the conformal fractional Laplacian is still true, i.e.,
$$P_\gamma[g^+,\hat g] f=\frac{d_\gamma}{2\gamma_0} A_m^{-1} \lim_{y\to 0}y^{a_0}\partial_y\left[y^{-1}\partial_y \lp y^{-1}\partial_y\lp\ldots y^{-1}\partial_y U\rp\rp \right],$$
where are $m+1$ derivatives in the formula above, and $U$ is the solution of the extension problem
\be\label{div1}\left\{\begin{split}
-\divergence \lp y^a \nabla U\rp + E(y) U &=0 \quad\mbox{in }(X,\bar g), \\
U&=f \quad\mbox{on }M;
\end{split}\right.\ee
here the derivatives are taken with respect to the metric $\bar g=y^2 g^+$, and the constants are $M=[\gamma]$, $\gamma_0=\gamma-m$, $a_0=1-2\gamma_0$, $A_m$ is given in \eqref{A_m} and $d_\gamma$ in \eqref{P-operator}.
\end{thm}

\begin{proof} The first assertion is a straightforward calculation from \eqref{Error}:  we know that near $\{y=0\}$, the metric $\bar g$ can be split as $\bar g=dy^2+g_y$. Then
$$\Delta_{\bar g}=\partial_{yy}+\tfrac{1}{2} \Psi\partial_y+\Delta_{g_y}$$
for $\Psi$ as given in \eqref{J}. Moreover, for the conformal change $g^+=y^{-2}\bar g$, we can write the scalar curvature equation
$$-\Delta_{\bar g}\lp y^{-\frac{N-2}{2}}\rp+ c_N R_{\bar g} y^{-\frac{N-2}{2}}=c_N \lp y^{-\frac{N-2}{2}}\rp^{\frac{N+2}{N-2}} R_{g^+}, \quad c_N= \frac{N-2}{4(N-1)},\quad N=n+1.$$
Since $R_{g^+}=-N(N-1)$ we quickly obtain that
$$R_{\bar g}=-n\Psi y^{-1}.$$
With all these ingredients, computing all the terms in \eqref{Error}, we can show that near $M$,
$$E(y)=\tfrac{-n+1+a}{4} \Psi y^{a-1}=\tfrac{n-1-a}{4n} R_{\bar g} y^a.$$

Now we compute the asymptotic behavior of $E(y)$ when $y\to 0$.
Note that for (even) Poincar\'e metrics, we have the expansion \eqref{taylor}, and thus,
$$\lim_{y\to 0} \frac{E(y)}{y^a}=\frac{-n+1+a}{4} \,\trace_{g^{(0)}} g^{(2)}.$$
\end{proof}

\begin{remark}\label{remark-curvature}
Before we continue, we remind the reader of how to compute the $Q_\gamma[g^+,\hat g]$ curvature, as defined in \eqref{Q-curvature}, for $\gamma\in\lp 0,\frac{n}{2}\rp\backslash \mathbb N$, $s=\frac{n}{2}+\gamma$. We set $f\equiv 1$, and find the solution to the Poisson problem $-\Delta_{g^+} v-s(n-s)v=0$ such that
$$v=Fy^{n-s}+H y^s,\quad F|_{y=0}=1,\quad  H|_{y=0}=h.$$
Then,
$$Q_\gamma[g^+,\hat g]=d_\gamma h.$$
\end{remark}

Next, we show that it is possible to find a special defining function satisfying that the zero-th order term $E(\rho)$ in equation \eqref{div} vanishes so that the extension problem is a pure divergence equation similar to the Euclidean one \eqref{equation-CS} studied in section \ref{extension}. We recover the conformal powers of the Laplacian as the Dirichlet-to-Neumann operator from theorem \ref{thm-zero} (or theorems \ref{thm-extension1}, \ref{thm-extension2}), plus a curvature term.

\begin{lemma}\label{lemma-special-defining}
Let $(X,g^+)$ be a conformally compact Einstein manifold with conformal infinity $(M,[\hat g])$. Fixed a metric $\hat g$ on $M$, assume that  $y$ is the defining function on $X$ such that on a neighborhood $M\times(0,\delta)$, the metric splits as
$g^+=\frac{dy^2+g_y}{y^2}$, where $g_y$ is a one-parameter family of metrics on $M$ with $g_y|_{y=0}=\hat g$, and Taylor expansion \eqref{taylor}. For each $\gamma\in(0,1)$, there exists another (positive) defining function $\rho^*$ on $M\times(0,\delta)$, satisfying  $\rho^*=y+O(y^{2\gamma+1})$, and such that for the term $E$ defined in \eqref{Error} we have
$$E(\rho^*)\equiv 0.$$
Moreover, the metric $g^*=(\rho^*)^2g^+$ satisfies $g^*|_{y=0}=\hat g$ and has asymptotic expansion
\be\label{asymptotics-g*}g^*=(d\rho^*)^2\left[ 1+O((\rho^*)^{2\gamma})\right]+\hat g\left[1+O((\rho^*)^{2\gamma})\right].\ee
\end{lemma}

\begin{proof}
We solve the eigenvalue problem \eqref{equation-GZ} with Dirichlet condition \eqref{general-solution} given by $f\equiv 1$, and $s=\frac{n}{2}+\gamma$. The solution can be written as $v=y^{n-s}F+y^s H$ for
$$F=1+O(y^2), \quad H=h+O(y^2).$$
We set
\be\label{rho*}\rho^*:=v^{1/(n-s)};\ee
we emphasize that $\rho^*$ is chosen as a power of the eigenfunction of $\Delta_{g^+}$ at $s=\frac{n}{2}+\gamma$.
We check now that this $\rho^*$ satisfies all the properties stated in the lemma.

First note that $v$ solves the equation $-\Delta_{g^+} v-s(n-s) v=0$. This quickly implies the vanishing of $E(\rho^*)$ if we use the equivalent formula \eqref{E1}. Next, the asymptotic behavior of $\rho^*$ is precisely
\be\label{asymptotics-rho*}\rho^*(y)=y\left[1+\tfrac{1}{n-s}h y^{2\gamma}+O(y^2)\right].\ee
Moreover, because of remark \ref{remark-curvature}, we precisely have $d_\gamma h=Q_\gamma[g^+,\hat g]$. On the other hand, the asymptotic expansion \eqref{asymptotics-g*}  for $g^*$ can be easily checked from \eqref{asymptotics-rho*} and the asymptotic behavior of $\bar g=dy^2+g_y(x)$, $g_y=\hat g+O(y^2)$.
\end{proof}

\begin{lemma} The $\rho^*$ constructed in the previous lemma can be defined in the whole $X$ and it is indeed positive and smooth.
\end{lemma}

\begin{proof} A summary of properties for degenerate elliptic equations, necessary to deal with \eqref{GZ} or \eqref{div} can be found in Gonz\'alez-Qing \cite{fractional-Yamabe}. The classical reference on degenerate elliptic equations with Muckenhoupt weights is Fabes-Kenig-Serapioni \cite{Fabes-Kenig-Serapioni:local-regularity-degenerate}, while Cabr\'e-Sire \cite{Cabre-Sire:I} has retaken the study in relation to fractional Laplacians on Euclidean space.

Note that $v$ is strictly positive in the whole $X$, thanks to the maximum principle and uniqueness for equation \eqref{GZ}. This shows that $\rho^*$ is an acceptable defining function on the whole $X$.
\end{proof}

\begin{thm}\label{thm-new-extension}
Assume the same hypothesis as in lemma \ref{lemma-special-defining}. Fix $\gamma\in(0,1)$ and let $\rho^*$ be the special defining function constructed in lemma \ref{lemma-special-defining}. For each smooth function $f$ on $M$, let $U$ solve the extension problem
\be\label{equation-extension}\left\{
\begin{split}
-\divergence \lp (\rho^*)^a \nabla U\rp&=0 \quad \mbox{in }(X, g^*),\\
U&=f \quad\mbox{on }M,
\end{split}\right.\ee
where the derivatives are taken with respect to the metric $g^*=(\rho^*)^2 g^+$. Then
\be\label{general-manifold-new}
P_\gamma[g^+,\hat g] f=\frac{d_\gamma}{2\gamma} \lim_{\rho^*\to 0} (\rho^*)^a\partial_{\rho^*} U+f Q_\gamma[g^+,\hat g].\ee
The fractional order curvature $Q_\gamma[g^+,\hat g]$ is defined in remark \ref{remark-curvature}.
\end{thm}

\begin{proof} We set up the same notation as in the previous results.
On one hand, in order to compute the scattering operator with Dirichlet data $f$ we need to consider the equation \be\label{formula100}-\Delta_{g^+} u-s(n-s)u=0 \quad \mbox{in }(X,g^+).\ee There exists a solution $u=y^{n-s} F+y^s H$
with $F=f+0(y^2)$, $H=h+O(y^2)$. Then $P_\gamma f:= P_\gamma[g^+,\hat g]f$ is defined as $P_\gamma f=d_\gamma h$.

On the other hand, the $Q_\gamma$ curvature can be computed as indicated in remark \ref{remark-curvature}. Let $v$ be the solution of the  Poisson equation  \eqref{formula100}, but with Dirichlet data data identically $\tilde f\equiv 1$. It can be written as
$$v=y^{n-s} \tilde F+y^s \tilde H \quad \mbox{with} \quad \tilde F=1+0(y^2), \quad \tilde H=\tilde h+O(y^2).$$
Then $Q_\gamma:=P_\gamma 1=d_\gamma \tilde h$. In addition, $\rho^*=v^{1/(n-s)}$.

Now we claim that $U:=(\rho^*)^{s-n} u$ satisfies \eqref{equation-extension}. First, lemma \ref{lemma-any-extension} applied to the defining function $\rho^*$ gives that $U$ is a solution of \eqref{div}. Second, for our special choice of defining function, $E(\rho^*)=0$ thanks to lemma \ref{lemma-special-defining}. Finally, note that by construction,
$$U=\frac{u}{v}=\frac{F+y^{2s-n}H}{\tilde F+y^{2s-n}\tilde H}.$$
In particular, $U|_{y=0}=f$. This shows that $U$ is a solution of \eqref{equation-extension}, as claimed.

Next, let us compute $\partial_y U$. It is easy to check that
$$\lim_{y \to 0}y^a\partial_y U=(2\gamma)h-2\gamma f\tilde h,$$
and \eqref{general-manifold-new} follows. This completes the proof of theorem \ref{thm-new-extension}.
\end{proof}

Before we state the next result, we need to introduce some notations: for any smooth function $w$, we denote $w=O_E(1)$ if the function $w$ has only even terms in the expansion (up to order $n$), i.e,
$$w=w_0+w_1y^2+w_2y^4+\ldots$$
We also define the operator $B:=y^{-1}\partial_y$. We claim that for any $k=1,2,\ldots$, it is true that:
\begin{enumerate}
\item If $w=O_E(1)$, then also $B^k w=O_E(1)$.
\item If $w=hy^{l}O_E(1)$ for some function $h=h(x)$ and $l\in\mathbb N$, then
$$B^k w=l(l-2)\ldots(l-2(k-1))y^{l-2k}h+O(y^{l-2k+1}).$$
and
$$\partial_y B^k w=l(l-2)\ldots(l-2(k-1))(l-2k)y^{l-2k-1}h+O(y^{l-2k}).$$
\end{enumerate}

We show now that theorem \ref{thm-new-extension} is is still valid for any exponent $\gamma\in \lp 0,\frac{n}{2}\rp\backslash \mathbb N$, thus generalizing theorem \ref{thm-extension2} on hyperbolic space to any conformally compact Einstein manifold:

\begin{thm} \label{thm-new-extension2}
Fix $s=\frac{n}{2}+\gamma$ for $\gamma\in\lp 0,\frac{n}{2}\rp\backslash \mathbb N$. Set $\gamma=m+\gamma_0$, $m=[\gamma]$, $\gamma_0\in(0,1)$, $a=1-2\gamma$.
In the hypothesis of theorem \ref{thm-new-extension}, let $U$ be a solution of \eqref{equation-extension}
Then
$$P_\gamma[g^+,\hat g] f=\frac{d_\gamma}{2\gamma_0}A_m^{-1}\left[\lim_{\rho^*\to 0}(\rho^*)^{a_0}\partial_{\rho^*} B^m U \right]+f Q_\gamma[g^+,\hat g],$$
where $a_0=1-2\gamma_0$ and $A_m$ is given in \eqref{A_m}.
\end{thm}

\begin{proof}
Let us compute
\bee Lim:=\lim_{y\to 0}y^{a_0}\partial_y (B^m U)\eee
for $U=(\rho^*)^{s-n} u$. As before, we can write
$$U=\frac{F+y^{2s-n}H}{\tilde F+y^{2s-n}\tilde H}=:\frac{V}{W},$$
where
$$F=f+y^2O_E(1), \tilde F=1+y^2O_E(1),\quad H=h+y^2O_E(1), \tilde H=\tilde h+y^2O_E(1).$$
We apply the product formula
$$B^m (V W^{-1})=\sum_{k=0}^m c_{m,k} (B^k V) (B^{m-k}(W^{-1})).$$
From this sum, only the terms $k=0$, $k=m$ are important, since the rest are of higher order in $y$.  Note that the term $k=m$ is just
\bee\begin{split}
(B^m V) W^{-1}&=(2s-n)(2s-n-2)\ldots(2s-n-2(m-1))y^{2s-n-2m}[h+y^2 O_E(1)]\\
&+O_E(1)+h.o.t.\end{split}\eee
so that
$$\partial_y \left[(B^m V) W^{-1}\right]=O(y)+2\gamma_0 A_m y^{2\gamma_0-1}h+h.o.t.$$
On the other hand, for the $k=0$ term we observe that
$$B^m(W^{-1})=-B^m(W)/W^2+h.o.t.,$$ and thus
$$\partial_y \left[V B^m(W^{-1})\right]=-y^{2\gamma_0-1}2\gamma_0 A_m\tilde h f+h.o.t.$$
Consequently
$$Lim=2\gamma_0A_m h-2\gamma_0A_m \tilde h f=\frac{2\gamma_0A_m}{d_\gamma}\left[P_\gamma f- f P_\gamma 1\right],$$
and the the proof of the proposition is completed after taking into account that $P_\gamma 1=Q_\gamma$.
\end{proof}


\section{The general case}

\setcounter{equation}{00}

Given a compact manifold $\bar X$ with boundary $M$, and a smooth metric $\bar g$ on $\bar X$, there exists an asymptotically hyperbolic metric with constant scalar curvature in the interior $X$, in the same conformal class of $\bar g$. This is the well known singular Yamabe problem, that has been well understood in a series of papers. We should cite Aviles-MacOwen \cite{Aviles-McOwen}, Mazzeo \cite{Mazzeo:regularity-singular-Yamabe}, Andersson-Chrusciel-Friedrich \cite{Andersson-Chrusciel-Friedrich}, in the case of negative constant scalar curvature.

We remark now that the construction of scattering operator $S(s)$ can be generalized to manifolds which are not Einstein, but just asymptotically hyperbolic (c.f. Mazzeo-Melrose \cite{Mazzeo-Melrose:meromorphic-extension} for most of the values of the parameter $s$, and Guillarmou \cite{Guillarmou:meromorphic} for a closer look at the remaining poles). In this section we try to understand how many of the results in the previous sections generalize to a compact manifold with boundary $(\bar X,\bar g)$, not necessarily conformally compact Einstein. For the rest of the section, $\gamma\in(0,1)$.

The case $\gamma=\frac{1}{2}$ was studied by Gillarmou-Gillop\'e \cite{Guillarmou-Guillope:determinant}. They considered the scattering operator of asymptotically hyperbolic manifold, and its relation to mean curvature. Note that $\gamma=\frac{1}{2}$ is a splitting point for the behavior of $P_\gamma$, as we will see in theorem \ref{thm-extension-general}.\\

Let $(\bar X,\bar g)$ be a smooth $(n+1)$-dimensional compact manifold with boundary $M^n=\partial X$. As we have mentioned above, there exists an asymptotically hyperbolic metric $g^+$ in the interior $X$, conformal to $\bar g$, and that has negative constant scalar curvature $R_{g^+}=-n(n+1)$. Moreover, $g^+$ has a very specific polyhomogeneous expansion. More precisely, let $\rho$ be a geodesic boundary defining function of $(\partial \bar X, \bar g)$, i.e.,
$$\bar g=d\rho^2+\bar g_\rho$$
for some one-parameter family of metrics $\bar g_\rho$ on $M$,  then we have that
\be\label{metric1}g^+=\frac{\bar g(1+\rho \alpha+\rho^n\beta)}{\rho^2},\ee
where $\alpha\in \mathcal C^{\infty}(\bar X)$, $\beta\in \mathcal C^\infty(X)$ and $\beta$ has a polyhomogeneous expansion
\be\label{beta}\beta(\rho,x)=\sum_{i=0}^\infty \sum_{j=0}^{N_i} \beta_{ij} \rho^i (\log \rho)^j\ee
near the boundary, $N_i\in\mathbb N\cup \{0\}$ and $\beta_{ij}\in\mathcal C^{\infty}(\bar X)$. Here we note that the $\log$ terms do not appear in the first terms of the expansion, and they can be ignore in our setting, because $\gamma\in(0,1)$ and we will not need such high orders.
We define
\be\label{rho-hat-rho}\frac{1}{\hat\rho^2}:=\frac{1+\rho \alpha+\rho^n\beta}{\rho^2},\ee
so that \eqref{metric1} is rewritten as
\be\label{metric11}g^+=\frac{\bar g}{\hat\rho^2},\ee

On the other hand, note that $\bar g_\rho$ may not only have even terms in its expansion. However, by the work of Graham-Lee \cite{Graham-Lee}, fixed the boundary metric $\hat g:= \bar g_\rho|_{\rho=0}=\bar g|_{M}$, we can find a boundary defining function $y=\rho+O(\rho^2)$ such that
\be\label{metric2}g^+=\frac{dy^2+g_y}{y^2}\ee
near $M$, where $g_y$ is a one-parameter family of metrics on $M$ such that $g_y|_{y=0}=\hat g$, with the regularity of $\rho \alpha+\rho^n\beta$. The main property of $g_y$ is that, if we make the expansion
$$g_y=g^{(0)}+g^{(1)} y+O(y^2),$$
then
\be\label{metric5}g^{(0)}=\hat g,\quad trace_{g^{(0)}} g^{(1)}=0.\ee
We set $\tilde g=dy^2+g_y$ so that
\be\label{metric10}g^+=\frac{\tilde g}{y^2}.\ee
The scattering operator for $(X,g^+)$ is computed as follows: first, solve the Poisson equation $-\Delta_{g^+}u-s(n-s)u=0$. For each $f\in\mathcal C^\infty(M)$, there exists a solution of the form
\be\label{metric7}u=y^{n-s}F+y^s H,\quad F=f+O(y^2), \quad H=h+O(y).\ee
Then, for $s=\frac{n}{2}+\gamma$, we define the conformal fractional Laplacian in this setting as
\be\label{metric6}P_\gamma[g^+,\hat g] f= d_\gamma h.\ee
In our case, we do have some log terms in the expansion \eqref{beta}. However, they do appear at order $n$, and consequently, they do not change the first terms in the asymptotic expansion for $u$.

We write
\be\label{hat-rho1}\hat \rho=y(1+y\phi+O(y^2))\ee
 for some $\phi\in\mathcal C^\infty(\partial X)$ which
 will be made precise later. Because the metrics \eqref{metric11} and \eqref{metric10} are equal, we can write $y^2\bar g=\hat \rho^2 \tilde g$. Then, restricting to the tangential direction, we get that
\be\label{metric3}(1+2y\phi+O(y^2)) g_y= \bar g_\rho.\ee
We write the Taylor expansions for $g_y$ and $\bar g_\rho$ in the variable $y$, taking into account that $\rho=y+O(y^2)$:
\be\label{Taylor-metric}\begin{split}
&g_y=g^{(0)}+g^{(1)} y+O(y^2), \quad g^{(0)}=\hat g, \quad g^{(1)}=0,\\
&\bar g_{\rho }=\bar g^{(0)}+\bar g^{(1)} y+O(y^2),\quad \bar g^{(0)}=\hat g,
\end{split}\ee
then looking at the orders of $y$ in \eqref{metric3}, we obtain
$$g^{(1)}+\phi \hat g=\bar g^{(1)}.$$
Taking trace above with respect to $\hat g$, and using \eqref{metric5}, we are able to find a formula for $\phi$. Indeed,
\be\label{phi}\phi=-\frac{1}{2n} trace_{\hat g} (\bar g^{(1)}).\ee
This is nothing but the mean curvature of $\partial X$ as a boundary of the $(n+1)$-manifold $(\bar X,\bar g)$, with a minus sign. We denote it by
\be\label{mean-curvature}\Psi_0:=\frac{1}{2n} trace_{\hat g} (\bar g^{(1)}).\ee\\

We have shown from \eqref{hat-rho1} and \eqref{phi} that
\be\label{hat-rho}\hat \rho=y(1-\Psi_0 y+O(y^2) ).\ee
We will need at a later section the relation between $\rho$ and $y$, so we indicate it here. First, from \eqref{rho-hat-rho} we know that
\be\label{rho-hat-rho2}\rho=\hat\rho\lp 1+\frac{\alpha}{2}\hat\rho+O(\hat\rho^2)\rp.\ee
If we substitute \eqref{hat-rho}, then
\be\label{y-rho}\rho=y\left[ 1+\lp -\Psi_0+\frac{\alpha}{2}\rp y+O(y^2)\right].\ee

In the following, we show in what cases lemma \ref{lemma-any-extension} for the calculation of $P_\gamma[g^+,\hat g]$ is still true:

\begin{thm}\label{thm-extension-general}
Let $(\bar X,\bar g)$ be a smooth, $(n+1)$-dimensional smooth, compact manifold with boundary, and let $\hat g$ be the restriction of the metric $\bar g$ to the boundary $M:=\partial X$. Let $\rho$ be a geodesic boundary defining function. Then there exists an asymptotically hyperbolic metric $g^+$ on $X$ of the form \eqref{metric1}, with respect to which the conformal fractional Laplacian  $P_\gamma[g^+,\hat g]$ can be defined as \eqref{metric6}, and it can be computed through the following extension problem:  for each smooth given function $f$ on $M$, consider
\be\label{div3}\left\{\begin{split}
-\divergence(\rho^a\nabla U)+E(\rho) &=0\quad\mbox{in }(\bar X,\bar g), \\
U&=f \quad \mbox{on }M,
\end{split}\right.\ee
where the derivatives are taken with respect to the original metric $\bar g$. Then, there exists a unique solution $U$ and moreover,
\begin{enumerate}
\item For $\gamma\in(0,\frac{1}{2})$,
\be\label{compute-fractional}P_\gamma[g^+,\hat g] f=\frac{d_\gamma}{2\gamma}\lim_{\rho\to 0}\rho^a \partial_\rho U,\ee

\item For $\gamma=\frac{1}{2}$, we have an extra term
$$P_{\frac{1}{2}}
[g^+,\hat g] f=\lim_{\rho\to 0} \partial_\rho U + \lp\tfrac{n}{2}-\tfrac{1}{2}\rp \Psi_0 f,$$
where $\Psi_0$ is the mean curvature of $M$ as defined in \eqref{mean-curvature}.

\item If $\gamma\in\lp\frac{1}{2},1\rp$, the limit in the right hand side of \eqref{compute-fractional} exists if and only if the the mean curvature $\Psi_0$ of the boundary $\partial X$ in $(\bar X,\bar g)$ vanishes identically, in which case, \eqref{compute-fractional} holds too.
\end{enumerate}
\end{thm}

\begin{proof}
We follow the notations in the previous paragraphs. It is possible to find an asymptotically hyperbolic metric $g^+$ on $X$, of the form \eqref{metric1}. We define the new functions $\hat \rho$ as \eqref{rho-hat-rho}, and $y$ satisfying \eqref{metric2}. Then we have seen that it is possible to define the scattering operator for the metric $(X,g^+)$, as follows: let $u$ be the solution of the eigenvalue problem $-\Delta_{g^+}u-s(n-s)u=0$, with given data $f$. Its asymptotic expansion is given in \eqref{metric7}, $u=y^{n-s}F+y^s H$, $F_{y=0}=f$, $H|_{y=0}=h$, while the scattering operator is just $P_\gamma=d_\gamma h$.

Second, lemma \ref{lemma-any-extension} applied to the defining function $\hat \rho$ shows that
$$U:=\hat \rho^{s-n} u$$
is the solution of \eqref{div3}, since $\bar g=\hat \rho^2 g^+$.

Next, we have shown in \eqref{hat-rho} that
$$\hat \rho=y(1-y \Psi_0+O(y^2)),$$
so that we can expand $U$ as
$$U=\left[f-(s-n)\Psi_0y+O(y^2)\right]+y^{2s-n}\left[h+O(y)\right].$$
Then
\be\label{equation20}y^a \partial_y U=-(s-n)\Psi_0 f y^a +(2s-n) h+o(1).\ee
Now we take the limit above when $y\to 0$. It is clear that the limit exists if and only if $a\geq 0$, unless $\Psi_0=0$. In particular, if $\gamma< \frac{1}{2}$, then
$$\lim_{y\to 0}y^a \partial_y U=(2s-n) h,$$
as desired. For the case $\gamma=\frac{1}{2}$, just note that $d_{\frac{1}{2}}=-1$.
\end{proof}

\begin{remark} The second conclusion of theorem \ref{thm-extension-general}, i.e., the case $\gamma=\frac{1}{2}$, was shown  by Guillarmou-Gillop\'e in \cite{Guillarmou-Guillope:determinant}.
\end{remark}



\bibliographystyle{abbrv}

\end{document}